\documentclass[reqno,12pt]{amsart}
\usepackage{amsmath, amsthm, amssymb,xcolor,geometry,hyperref,enumerate,cite,csquotes}
\usepackage{verbatim}
\renewcommand{\d}{\partial}
\newcommand{\dbar}{\overline{\partial}}
\newcommand{\ddbar}{\sqrt{-1}\d\overline{\d}}

\newtheorem{thm}{Theorem}
\newtheorem{prop}[thm]{Proposition}
\newtheorem{lem}[thm]{Lemma}
\newtheorem{cor}[thm]{Corollary}

\newtheorem{rem}[thm]{Remark}

\theoremstyle{definition}
\newtheorem{case}{Case}

\renewcommand{\[}{\begin{equation}}
	\renewcommand{\]}{\end{equation}}

\newcommand{\al}{\alpha}
\newcommand{\be}{\beta}

\newcommand{\la}{\lambda}

\newcommand{\La}{\Lambda}

\newcommand{\vp}{\varphi}
\newcommand{\vep}{\varepsilon}

\newcommand{\RR}{\mathbb{R}}
\newcommand{\CC}{\mathbb{C}}

\newcommand{\PP}{\mathbb{P}}
\newcommand{\sH}{\mathcal{H}}

\numberwithin{equation}{section}
\counterwithin{thm}{section}
\counterwithin*{case}{thm}
\counterwithin*{clm}{thm}

\hypersetup{
	colorlinks=true,
	linkcolor=blue,
	citecolor=red
}

\title{A note on modified $J$-Flow with the Calabi Ansatz}
\author[P. Sivaram]{P. Sivaram}
\address{Department of Mathematics, Indian Institute of Science, Bangalore, India - 560012}
\email{}
\allowdisplaybreaks

\begin{document}
\maketitle
\begin{abstract}
	We study the modified $J$-flow introduced in \cite{lishi2}, particularly the singularities of the flow using the Calabi symmetry. In \cite{tak}, on toric manifolds the convergence of modified $J$-flow to the smooth solution was proven under the assumption of positivity of certain intersection numbers. In the case of the Calabi ansatz we show that if some of those intersection numbers are not positive, then the modified $J$-flow blows up along some variety and away from the variety we prove the convergence to the solution.
	
	As in \cite{fanglai}, we also prove that the convergence behavior of the modified $J$-flow with Calabi symmetry depends on the topological constants $c$ and the minimum of the Hamiltonian function.
\end{abstract}
\section{Introduction}
There has been considerable progress in recent years on solvability of inverse Hessian equations going beyond the more classically studied complex Monge-Ampere equations. An equation that has been of particular interest, in part due to its relevance to constructing constant scalar curvature K\"ahler (cscK) metrics, is the $J$-equation. The $ J $-equation on K\"ahler manifolds was introduced by Donaldson \cite{don} in the moment map setting and by X. Chen \cite{jeqn} as the Euler-Lagrange equation for a functional that appears in his formula for the Mabuchi functional.  Let $(X,\omega)$ be a K\"ahler manifolds, and let $\chi$ be another K\"ahler metric. The $J$-equation seeks a K\"ahler metric $\chi_\varphi:= \chi+ \ddbar\varphi$ in the class of $[\chi]$ satisfying $$\La_{\chi_\varphi}\omega:= n\frac{\omega\wedge\chi_\vp^{n-1}}{\chi_\vp^n} = c,$$ where $c$ is necessarily given by $$c = n\frac{[\chi]^{n-1}\cdot[\omega]}{[\chi]^n}.$$  A consequence of Chen's formula for the Mabuchi energy (cf. \cite{swflow}) is that solvability of the $J$-equation in the class $[\chi] = c_1(K_M)$ implies existence of constant scalar curvature metrics.  

In \cite{bisec}, X. Chen introduced a natural flow, the so-called $J$-flow, to study existence of solutions to the $J$-equation and proved it's long-time existence and convergence under a bi-sectional curvature lower bound. The $ J $-flow is defined as follows:
\[
\begin{cases}
	\frac{\d \vp}{\d t}=c-\La_{\chi_\varphi}\omega\\
	\vp_t|_{t=0}=\vp_0 \in \sH,
\end{cases}
\]
where $ \sH=\{\vp:\chi_\vp=\chi+\ddbar\vp>0\}. $ Note that on K\"ahler surfaces, the $J$-equation reduces to a complex Monge-Ampere equation, and hence by Yau's celebrated resolution of the Calabi conjecture (cf.  \cite{yau}), a solution exists if and only if $[c\chi-\omega]$ is a K\"ahler class. More generally, in \cite{swflow}, Song and Weinkove proved that a solution to the $J$-equation in the class $[\chi]$ exists if and only if  there exists a  $ \vp\in \sH $ such that  $$ c\chi_\vp^{n-1}-(n-1)\omega\wedge\chi_\vp^{n-2}>0. $$ 

The interested readers can refer to \cite{flowsurface,flowhi} for earlier results, and \cite{Sz-pde} for an extension of this theorem to more general inverse Hessian equations. While this is an optimal PDE result, the condition above (called the cone condition) is a {\em pointwise} condition, and is generally difficult to verify. Inspired by the work of Demailly and Paun \cite{DemPa} and the close analogy between complex Monge-Ampere equations and the $J$-equation, Lejmi and Szekelyhidi \cite{stab} conjectured that a solution to the $J$-equation exists if and only if the following Nakai type criteria holds: $$ \int_V(c\chi^p-p\omega\wedge\chi^{p-1})>0, $$ for any $ p $-dimensional sub-variety $ V $ of the manifold. Note that the above condition is a numerical or topological condition, and at least in principle, much easier to verify. A uniform version of this conjecture was proved by G. Chen in \cite{gchen}. Based on the work of Chen, the full Lejmi-Szekelyhidi conjecture was established by Datar-Pingali  \cite{DP} on  projective manifolds and Song \cite{So} for general K\"ahler manifolds. 

In analogy with the role of the $J$-equation in studying the Mabuchi functional, the modified $ J $-equation was introduced by Li-Shi \cite{lishi2} to study  the modified Mabuchi functional. The setting is as follows: Suppose now, that $ \omega$ and $\chi $ are invariant under the action of a real torus $ T\subset Aut_0(X) $. Let $$ \sH^T=\{\vp\in C^\infty(X,\RR)^T:\chi_\vp=\chi+\ddbar\vp>0\} $$ be the $ T $-invariant K\"ahler potentials of $ \chi. $ Let $ \xi $ be a holomorphic vector field with $ Im(\xi)\in \mathfrak{t}. $ We define the Hamiltonian function of $ \xi $ with respect to the metric $ \chi $ as the real valued function uniquely determined by the following properties:
$$ i_\xi\chi=\frac{\sqrt{-1}}{2\pi}\dbar\theta_\xi(\chi) \text{ and } \int_{X}\theta_\xi(\chi)\chi^n=0. $$ 
The modified $ J $-equation is defined as \[\label{modjeqn} n\frac{\omega\wedge\chi_\vp^{n-1}}{\chi_\vp^n}=c+\theta_\xi(\chi_\vp), \] where $c$ is as before. Note that the above equation reduces to the $J$-equation if $ T $ is trivial.
In \cite{lishi2} Li-Shi proved that if there exist a $ \hat{\chi}\in [\chi] $ such that $$ (c+\theta_\xi(\hat{\chi}))\hat{\chi}^{n-1}-(n-1)\omega\wedge\hat{\chi}^{n-2}>0, $$ then there exist a unique $ \vp\in \sH^T $ which satisfies the modified $ J $-equation \eqref{modjeqn}. They proved the result using the parabolic flow method by defining the modified $ J $-flow by adapting the arguments in \cite{swflow}.The modified $ J $ flow is defined as 
\begin{equation}\label{modjflow}
	\begin{cases}
		\frac{\d \vp}{\d t}=c+\theta_\xi(\chi_\vp)-n\frac{\omega\wedge \chi_\vp^{n-1}}{\chi_\vp^n}\\
		\vp_t|_{t=0}=\vp_0 \in \sH^T.
	\end{cases}
\end{equation} 
In \cite{tak} Takahshi conjectured that the modified $ J $-equation has solution in the class $[\chi]$ if and only if $c+\theta_\xi(\chi)>0$ and the following Nakai criteria holds: For all $p$-dimensional toric sub-varieties $V\subset M$, $$ \int_V((c+\theta_\xi(\chi))\chi^p-p\omega\wedge\chi^{p-1})>0,$$ and verified this conjecture for toric manifolds (cf. \cite{CoSz} for an analogous result for the $J$-equation). 

The main goal of this note is to study the behaviour of the modified $J$-flow on the blow-ups $X_n = \mathrm{Bl}_{x_0}\PP^n$ of $\PP^n$. By Takahashi's result, the existence problem is completely settled, and so we focus instead on the case when the Nakai criteria {\em fails}. These are also examples of ruled surfaces. Indeed $X_n = \PP(\mathcal{O}_{\PP^{n-1}}(-1)\oplus \mathcal{O})$, and one can study the modified $J$-flow using the Calabi ansatz. For the $J$-flow, Fang and Lai  obtained a complete description of the behaviour of the $J$-flow on these manifolds using the Calabi ansatz in \cite{fanglai}. In particular, that even when the Nakai criteria fails, the flow still converges on the complement of the exceptional divisor $E$ to a solution of the $J$-equation on $X_n\setminus E$ with a different slope. Our main goal is to extend these results to the modified $J$-flow on these manifolds.

To state our main result, we introduce some notation. Given real numbers $a,b>1$ and $k\geq 0$, we let $$ c_k=n\frac{ab^{n-1}-1}{b^n-1}-\frac{nk}{n+1}\frac{b^{n+1}-1}{b^n-1}. $$ When $k\geq 0$, we characterise the behaviour of the modified $J$ flow using the constant $c_k+k$. 
\begin{thm}\label{mt1}
	Let $ X_n=\PP^n\#\overline{\PP^n} $ be the blow up of $ \PP^n $ at one point with two K\"ahler metrics $ \omega \in a[E_\infty]-[E_0] $ and $ \chi\in b[E_\infty]-[E_0] $, where $ E_0 $ and $ E_\infty $ be the exceptional divisor and the pull-back of the divisor associated to $ \mathcal{O}_{\PP^n}(1) $ respectively and $\xi=kw\frac{\d}{\d w}$,where $w$ is the fiberwise coordinate of $\mathcal{O}_{\PP^{n-1}}(-1)$ and $k\geq 0$. And let $ \chi_t $ be the solution of the modified $ J $-flow.
	Then the following three cases characterize the convergence behavior of the modified $ J $-flow:
	\begin{enumerate}
		\item If $ c_k+k>n-1 $, then $ \chi_t\rightarrow \chi_\infty $ as $ t\to \infty $ on $ X_n $ smoothly and $ \chi_\infty $ is the solution of the modified $ J $ equation.
		\item If $ c_k+k=n-1, $ then $ \chi_t\rightarrow \chi_\infty $ as $ t\to \infty $ on $ X_n\smallsetminus E_0 $ smoothly, where $ \chi_\infty=\chi+\ddbar\vp,\text{ for some } \vp \in L^\infty(X_n), $ a singular K\"ahler metric has a conical singularity along $ E_0 $ and smooth everywhere else. And the flow converges to the solution of the equation $$ n\frac{\omega\wedge\chi_\infty^{n-1}}{\chi_\infty^n}=c+\theta_\xi(\chi_\infty), \text{ on }X_n\smallsetminus E_0. $$
		\item If $ 0<c_k+k<n-1 $, then $ \chi_t\to \hat{\chi}_\infty+(\la-1)[E_0] $, a K\"ahler current as $ t\to \infty $ on $ X_n\smallsetminus E_0 $ smoothly, where $ \la \in (1,b) $ is unique such that $$ n\frac{ab^{n-1}-\la^{n-1}}{b^n-\la^n}-\frac{nk}{n+1}\frac{b^{n+1}-\la^{n+1}}{b^n-\la^n}+k\la=\frac{n-1}{\la}. $$
		And the flow converges to the solution of the equation $$ n\frac{\omega\wedge\hat{\chi}_\infty^{n-1}}{\hat{\chi}_\infty^n}= n\frac{ab^{n-1}-\la^{n-1}}{b^n-\la^n}+\theta_\xi(\hat{\chi}_\infty), \text{ on } X_n\smallsetminus E_0. $$
	\end{enumerate}
\end{thm}

To conclude this section, we make a few remarks on some ongoing work, and possible future directions. Firstly we note that in a work in progress, we introduce a notion of Futaki invariant for the modified $J$-equation, and prove a Atiyah-Bott type lower bound for an $L^2$-energy in terms of the Futaki invariant, obtaining a modified $J$-equation analogue of a similar result due to Lejmi and Szekelyhidi. As in \cite{stab}, using the above convergence result, we can improve the lower bound to an equality in the case of the manifolds $X_n$.  In another direction, one can also obtain a similar convergence result for more general projective bundles where the zero section may have a higher co-dimension. In \cite{vrj}, the authors develop a program for obtaining weak solutions to the $J$-equation even when the Nakai criteria fails, and obtain weak solutions on K\"ahler surfaces. It would be interesting to develop an analogous program for the modified $J$-equation. Unlike in the case of the $J$-equation, finding smooth solutions to the modified equation on K\"ahler surfaces is already non-trivial (cf. \cite{lishi2}). In light of this, one would expect additional difficulties in adapting the arguments in \cite{vrj} to the modified setting. 

\section{Modified $ J $-flow with the Calabi Ansatz}
\subsection{Modified $J$-equation with the Calabi Ansatz:}
We first review the ansatz of Calabi  to construct K\"ahler metrics on certain ruled surfaces. The interested reader can refer to \cite{Calabi, HS, gabor-book} for more details. Let $ X_n $ denote blow-up of  $ \PP^n $ at one point. Equivalently, $ X_n=\PP(\mathcal{O}_{\PP^{n-1}}(-1)\oplus\mathcal{O}) $. Let $ h $ be a metric on $ \mathcal{O}_{\PP^n}(-1) $ with curvature $ -\sqrt{-1}2\pi \omega_{FS} $, and write $ s=\log |.|_h $ for the log of the fiberwise norm. Denote by $ [E_0] $ and $ [E_\infty] $ the zero section (exceptional divisor) and the infinity section respectively. The K\"ahler cone is then given by $$ \mathcal{K}_{X_n}= \{ \be [E_\infty]-\al [E_0]: \be >\al>0 \} .$$ For an appropriate choice of strictly increasing convex function $ f \in \mathcal{C}^\infty(\RR) $ we can write down a K\"ahler metric $$ \omega=\ddbar f(s) $$
on $ X_n\smallsetminus(E_0\cup E_\infty) \cong \CC^n\smallsetminus\{0\}. $ At a point choose local coordinates $ z=(z_1,z_2,\dots,z_{n-1}) $ on $ \PP^{n-1} $ and a fiberwise coordinate $ w $ such that $ d\log h(z)=0. $  At this point we then have that $$ \omega=\sqrt{-1}f'(s)\omega_{FS}+f''(s)\frac{\sqrt{-1}dw\wedge d\bar{w}}{|w|^2}. $$
In order to extend $ \omega $ to a smooth K\"ahler metric on $ X_n $, the following asymptotic properties of $ f $ are required:
\begin{enumerate}[(i)]
	\item $ F_0(t):=f(\log t)-\al\log t $ extends to a smooth function at $ t=0 $, and $ F_0'(0)>0. $
	\item $ F_\infty(t):=f(-\log t)+ \be \log t $ extends to a smooth function at $ t=0 $, and $ F_\infty'(0)>0. $
\end{enumerate}
By the asymptotic behavior of $ f, $ we have $$ \lim_{s \to -\infty}f'(s)=\al, \lim_{s \to \infty}f'(s)=\be, $$
and the K\"ahler class of $\omega $ is then given by $\be [E_\infty]-\al [E_0]. $

We can repeat this construction for a different convex function $ g \in \textbf{C}^\infty(\RR) $ and consider $$ \chi=\ddbar g(s). $$
We normalize the K\"ahler classes so that $$ \omega \in a[E_\infty]-[E_0], \chi \in b[E_\infty]-[E_0] ;\quad a,b>1. $$
We now write down the modified $J$-equation using Calabi ansatz.
We introduce a moment map coordinate $ \tau=f'(s), $ and define the strictly increasing function $ \psi\colon[1,b] \to [1,a] $ by letting $$ \psi(f'(s))=g'(s), \forall~s\in \RR. $$ 

We consider the $ \mathbb{S}^1 $ action on $ X_n $ of the form
$$ e^{\sqrt{-1}\theta}\cdot (z,w)=(z,e^{\sqrt{-1}k \theta}w). $$
The vector field generating the action is given by $ \xi=kw\frac{\d}{\d w}. $ Since  $$ i_{\xi} \chi=i_{\xi} \left(f''(s)\frac{\sqrt{-1}dw\wedge d\bar{w}}{|w|^2}\right)=\sqrt{-1}~\dbar (kf'(s)), $$  the normalized Hamiltonian function in the momentum coordinate is $$ \theta_\xi(\chi)=k\left( \tau-\frac{n}{n+1} \frac{b^{n+1}-1}{b^n-1}\right) $$
then the modified $ J $ equation can be written as 
\[\label{jeqn}  \psi'(\tau)+(n-1)\frac{\psi(\tau)}{\tau}=c_k+k\tau,  \]
where $$ c_k=c-\frac{kn}{n+1}\frac{b^{n+1}-1}{b^n-1}  \text{ and }  c=n\frac{ab^{n-1}-1}{b^n-1}. $$ A solution must be strictly increasing and also satisfy the boundary conditions $ \psi(1)=1 $ and $ \psi(b)=a. $ If we take the derivative of the equation \eqref{jeqn} with respect to $\tau$ then we get a second order ODE associated to the modified $J$ equation
\begin{equation}\label{ode}
	\psi''(\tau)+(n-1)\frac{\psi'(\tau)}{\tau}-(n-1)\frac{\psi(\tau)}{\tau^2}-k=0.
\end{equation}
With the boundary conditions $ \psi(1)=1 $ and $ \psi(b)=a, $ we obtain a unique solution 
\begin{equation}\label{psi tilde}
	\tilde{\psi}(\tau)=\frac{k\tau^2}{n+1}+\frac{c_k }{n}\tau +\left(\frac{n}{n+1}-\frac{c_k}{n}\right)\frac{1}{\tau^{n-1}}.
\end{equation}
We will use the equation \eqref{ode} to define the modified $J$-flow and we also see that the presence of the quadratic term in the solution \eqref{psi tilde} makes the analysis of the equation different from the $J$-flow ODE. Note that if we take $k=0$, then we recover the $J$-equation.

From now onwards we assume that $ k\geq 0 $.
\begin{rem}\label{the constant}
	\begin{enumerate}[(a)]
		\item For the $J$-equation with the Calabi ansatz, solvability of the equation depends on the topological constant $c=n\frac{[\omega]\cdot[\chi]^{n-1}}{[\chi]^n}$. Analogous to that for the modified $J$-equation with Calabi ansatz, solvability depends on the minimum of the normalized hamiltonian and the constant $c$, that is we get a necessary and sufficient condition for the existence of a smooth solution to the modified $J$-equation from $$\int_{E_0}\left((c+\theta_\xi(\chi))\chi^{n-1}-(n-1)\omega\wedge \chi^{n-2}\right)=c_k+k-(n-1)>0 $$ and $$ c_k+k=\min_{X_n}\big(c+\theta_\xi(\chi)\big).$$
		\item By Lemma 2.1 of \cite{tak}, the constant $\min_{X_n}(c+\theta_\xi(\chi))$ is independent of the metric $\chi$ we pick from the class $[\chi]$ in general.
	\end{enumerate}
\end{rem}
\begin{thm}\label{critical}
	Let $X_n$ be a blow up of $ \PP^n $ at one point and  $ \omega \in a[E_\infty]-[E_0]$ and $ \chi \in b[E_\infty]-[E_0] $ to be two K\"ahler classes. If  $ c_k+k >n-1, $ then there exist a solution to the modified $ J $-equation \eqref{ode}.
\end{thm}
\begin{proof}
	The solution of the ODE \eqref{ode} is given by \[\label{soln jeqn} \psi(\tau)=\frac{k\tau^2}{n+1}+\frac{c_k}{n}\tau+\left(\frac{n+1-k}{n+1}-\frac{c_k}{n}\right)\frac{1}{\tau^{n-1}}. \]
	Given that $ c_k>n-(k+1). $ Then
	\begin{align*}
		\psi'(\tau)&=\frac{2k\tau}{n+1}+\frac{c_k}{n}-(n-1)\left(\frac{n+1-k}{n+1}-\frac{c_k}{n}\right)\frac{1}{\tau^n}\\
		&>\frac{2k}{n+1}+\frac{n-(k+1)}{n}-(n-1)\left(\frac{n+k+1}{n(n+1)}\right)\\
		&= 0.
	\end{align*}	
	This implies that $ \psi $  has an inverse. So, $ f'(s)=\psi^{-1}(g'(s)), \text{ for all } s \in \RR. $ Since $ f $ satisfies the properties of the Calabi ansatz we can get a metric which solves the modified $ J $-equation \eqref{modjeqn}.
\end{proof}
As a consequence we have the following rotationally symmetric version of Conjecture 1.4 of \cite{tak}.
\begin{cor}
	\label{iffcondition}Let $ X_n $ be the blow up of $ \PP^n $ at one point and $ \omega \in a[E_\infty]-[E_0]$ and $ \chi \in b[E_\infty]-[E_0] $ to be two K\"ahler classes. Then the following are equivalent 
	\begin{enumerate}
		\item The modified $J$-equation \eqref{modjeqn} has a solution.
		\item 	$c+\theta_\xi(\chi)>0$ and for each $ p\in \{1,2,\dots,n-1\} $ and $ p$-dimensional irreducible sub-variety $ V^p \subset X $, $$ \int_{V^p}\left((c+\theta_\xi(\chi))\chi^p-p\omega\wedge\chi^{p-1}\right)>0. $$
		\item $ c_k+k>n-1. $
	\end{enumerate}
\end{cor}
\begin{proof}	
	For $(1)$ implies $(2)$, by assumption there exist $ \chi \in b[E_\infty]-[E_0] $ which solves the modified $ J $-equation \eqref{modjeqn}
	$$ n\frac{\omega\wedge \chi^{n-1}}{\chi^n}=c+\theta_\xi(\chi). $$
	Let us choose a coordinate at a point $ y\in X_n $ such that $ \omega_{ij}=\delta_{ij} $ and $ \chi_{ij}=\mu_i\delta_{ij}, \mu_i>0 \text{ for } 1\leq i,j\leq n. $
	At this point the modified $ J $-equation will be of the form 
	$$ \sum_{i=1}^{n}\frac{1}{\mu_i}=c+\theta_\xi(\chi). $$
	This implies that $c+\theta_\xi(\chi)>0$, since $y\in X_n$ was arbitrary and for any $ \{i_1,i_2,\dots,i_p\} \subset \{1,2,\dots,n\} $ we have 
	$$ \sum_{r=1}^{p}\frac{1}{\mu_{i_r}}< c+\theta_\xi(\chi) $$
	So $ (c+\theta_\xi(\chi))\chi^p-(p-1)\omega\wedge\chi^{p-1}>0. $
	And if we integrate the form $ (c+\theta_\xi(\chi))\chi^p-(p-1)\omega\wedge\chi^{p-1} $ over any $ p $-dimensional irreducible sub-variety $ V^p $ we get $$ \int_{V^p} (c+\theta_\xi(\chi))\chi^p-(p-1)\omega\wedge\chi^{p-1}>0. $$
	For $(2)$ implies $(3)$, we have $ [E_\infty]|_{E_0}\equiv0 $ and $ [E_0]|_{E_0}\equiv -\omega_{FS} $ and we can take $ V^{n-1}=E_0. $ Then 
	\begin{align*}
		\int_{E_0}\left((c+\theta_\xi(\chi))\chi^{n-1}-(n-1)\omega\wedge\chi^{n-1}\right)&=\int_{E_0}\left((c_k+k)\chi^{n-1}-(n-1)\omega\wedge\chi^{n-2}\right)\\
		&=c_k+k-(n-1).
	\end{align*}
	Thus $ c_k+k>n-1 $ follows from the given hypothesis.
	And $(3)$ implies $(1)$ follows from Theorem \ref{critical}.
\end{proof}
\begin{rem}
	\begin{enumerate}[(a)]
		\item An important point to note here is that for the $ J $-equation if $ \omega $ and $ \chi $ belong to the same K\"ahler class, then we can always solve the $ J $-equation. But this is not the case for the modified $ J $-equation. 
		For example on $ X_2 $ with $ k=1 $, there is no $ \omega,\chi\in 3[E_\infty]-[E_0] $ such that $$ 2\frac{\omega\wedge\chi}{\chi^2}=2+\theta_\xi(\chi). $$ 
		Since $ c_1=\frac{-1}{6}<0 $ in this case.
		Also consider $ \omega\in 7[E_\infty]-[E_0] $ and $ \chi \in 6[E_\infty]-[E_0] $  in the blow up of $ \PP^2 $ at a point. Then $ c_1<0. $ Thus the modified $ J $-equation can not have a solution in $ [\chi] $. Hence $a\geq b$ does not guarantee the solution for the modified $J$ equation, unlike the $J$-equation.
		\item If $(\omega,\chi)$ is $J$-unstable, that is $c<n-1$, then we can not have solution for the modified $J$-equation on that particular pair for any $k\geq 0$, since $$k\left(1-\frac{n}{n+1}\frac{b^{n+1}-1}{b^n-1}\right)<0.$$
		\item When $ k\geq 0, $ we have $ c_k+k>n-1 $ as a necessary and sufficient condition for the existence of the solution of the modified $ J $-equation. This provides the upper bound on the value of $ k $, that is $$ 0\leq k < \left(1+n\frac{ab^{n-1}-1}{b^n-1}-n\right)\left(\frac{n}{n+1}\frac{b^{n+1}-1}{b^n-1}-1\right)^{-1}. $$
		\item  If we choose the $ k $ to be negative, then to get a sufficient condition for the existence of the increasing function we have to integrate over the sub-variety $ E_\infty. $
	\end{enumerate}
\end{rem}

Now we turn our attention to the unstable case, that is when $c_k+k<n-1.$
In this case the general solution\eqref{soln jeqn} of equation \eqref{ode} is not increasing in $[1,b]$.
Instead we consider the following auxiliary family of equations
\begin{equation}\label{sjeqn}
	\begin{cases}
		\psi''(\tau)+(n-1)\frac{\psi'(\tau)}{\tau}-(n-1)\frac{\psi(\tau)}{\tau^2}-k=0, &\tau\in (s,b)\\
		\psi(s)=1 \text{ and } \psi(b)=a, &\text{ where } s \in (1,b).
	\end{cases}
\end{equation}
Solution of the above equation is $$ \psi_s(\tau)=\frac{\tau^2}{n+1}+\frac{C_k(s)}{n}\tau+\frac{A_k(s)}{\tau^{n-1}}. $$ Using the boundary conditions to solve for $ C_k(s) \text{ and } A_k(s) $, we get \[\label{lam} C_k(s)=n\frac{ab^{n-1}-s^{n-1}}{b^n-s^n}-\frac{kn}{n+1}\frac{b^{n+1}-s^{n+1}}{b^n-s^n} \text{ and } A_k(s)=s^{n-1}\frac{n+1-ks^2}{n+1}-s^n\frac{C_k(s)}{n}. \]

Let 
$$	\mathcal{S}= \{ s \in (1,b):\text{ there exist } \psi_s \text{ solves } \eqref{sjeqn} \text{ with } \psi_s'(\tau)>0, \text{ for all }\tau\in[s,b] \}, $$ and $$ \la=\inf\mathcal{S}. $$

Since the solution \eqref{soln jeqn} $ \psi $ of the ODE \eqref{ode} is convex and $ \psi'(1)<0 $ in the case $c_k+k<n-1 $ there exist a unique $ \bar{s} $ such that $ \psi'(\bar{s})=0 $ and there exist a unique $ s\in (\bar{s},b) $ such that $ \psi(s)=1. $ Hence by definition $ s $ belong to the set $ \mathcal{S} $ and $ 1<\la\leq s $.

Then $ \la $ satisfies the following equation $$ C_k(\la)+k\la=\frac{n-1}{\la} $$ which follows from the equation $\psi_\la'(\la)=0$, by the definition of $\la.$

The left hand side of the above  equation (in analogy to the $J$-equation) is related to the topological constant of the modified $J$-equation as follows:

Let $\chi_\la\in b[E_\infty]-\la[E_0]$ and $\omega\in a[E_\infty]-[E_0]$. Then $$C_k(\la)+k\la=\min_{X_n}\left( n\frac{[\omega]\cdot[\chi_\la]^{n-1}}{[\chi_\la]^n}+\theta_\xi(\chi_\la)\right). $$

Suppose for $[\chi]=b[E_\infty]-[E_0], [\omega]=a[E_\infty]-[E_0]$ and vector field $\xi=kw\frac{\d}{\d w}(k\geq 0)$, we get that $0<c_k+k<n-1$, we claim that we can perturb the K\"ahler class of $\chi$ so that we can solve the modified $J$-equation in the perturbed class. And that perturbation is measured by the constant $\la$ that we introduced above.

For any $s\in (1,b)$, let $\chi_s\in b[E_\infty]-s[E_0]$ be any K\"ahler metric and we know that $$ C_k([\chi_s],[\omega])+ks=\min_{X_n}\left(n\frac{[\omega]\cdot[\chi_s]^{n-1}}{[\chi_s]^n}+\theta_\xi(\chi_s)\right). $$
\subsection{Geometric description of the constant $\la$ in the case $c_k+k<n-1$:}
In this subsection we give a geometric description of the constant $\la$.

Let $ \chi_s \in  b[E_\infty]-s[E_0] $
and $$\mathcal{S}=\left\{s\in(1,b):C_k([\chi_s],[\omega])+ks>\frac{n-1}{s}\right\}. $$ 
Let us look at the constant $C_k([\chi_s],[\omega])+ks$,
$$C_k([\chi_s],[\omega])+ks=n\frac{ab^{n-1}-s^{n-1}}{b^n-s^n}-\frac{nk}{n+1}\frac{b^n+b^{n-1}s+\cdots+s^{n-1}}{b^{n-1}+b^{n-2}s+\cdots+s^{n-1}}+ks. $$
As $s\to b$ the first term in the above sum goes to infinity and the second and third term are bounded for all $s\in (1,b)$, the set $\mathcal{S}$ is nonempty.

We claim that the condition $C_k([\chi_s],[\omega])+ks>\frac{n-1}{s}$ is the condition for the existence of the solution of the modified $J$-equation in the K\"ahler class $[\chi_s]$.
Let $\al_s\in b_1[E_\infty]-[E_0], $ where $b_1=bs^{-1}$ for $s\in \mathcal{S}$ be any K\"ahler metric, $k_1=s^2k$ and $\xi'=k_1w\frac{\d}{\d w}=s^2\xi$. Then $$ C_k([\chi_s],[\omega])+ks>\frac{n-1}{s}\text{ if and only if } C_{k_1}([\al_s],[\omega])+k_1>n-1. $$
By Theorem \ref{critical}, there exist $\hat{\al}_s\in [\al_s]$ such that \[\label{hatalpha} n\frac{\omega\wedge\hat{\al}_s^{n-1}}{\hat{\al}_s^n}=n\frac{[\omega]\cdot[\hat{\al}_s]^{n-1}}{[\hat{\al}_s]^n}+\theta_{\xi'}(\hat{\al}_s). \]
Letting $\hat{\chi}:=s\hat{\al}_s\in b[E_\infty]-s[E_0]$, if we multiply the equation \eqref{hatalpha} by $s^{-1}$ we get $$n\frac{\omega\wedge\hat{\chi}^{n-1}}{\hat{\chi}^n}=n\frac{[\omega]\cdot[\hat{\chi}]^{n-1}}{[\hat{\chi}]^n}+\theta_\xi(\hat{\chi}) ,$$
which proves our claim. In particular by the definition of $\la$ we have  $$ C_k(\la)+k\la=\frac{n-1}{\la} \text{ and }  \la=\inf\mathcal{S}.  $$
Note that since $c_k+k<n-1$ by Corolloary \ref{iffcondition} we also have that $\la>1.$

Moreover, if we take any $s\in (1,\la)$, then by definition of $\la$ we get that $$ C_k([\chi_s],[\omega])+ks< \frac{n-1}{s} $$ and this is equivalent to $$  C_{k_1}([\al_s],[\omega])+k_1<n-1. $$ 
Thus we can not have a smooth solution to the modified $J$-equation in the class $[\chi_s].$

\subsection{Modified $J$-flow with the Calabi ansatz:}
 Recall that the modified $ J $-flow is defined as 
\begin{equation}\label{j flow}
	\begin{cases}
		\frac{\d \vp}{\d t}=c-n\frac{\omega\wedge \chi_\vp^{n-1}}{\chi_\vp^n}+\theta_\xi(\chi_\vp)\\
		\vp_t|_{t=0}=\vp_0 \in \sH^T,
	\end{cases}
\end{equation} 
where $ \theta_\xi(\chi_\vp) $ is the Hamiltonian function of the action of $ \xi $ with respect to the metric $ \chi_\vp $ which is normalized so that $ \int_X\theta_\xi(\chi_\vp)\chi_\vp^n=0 $ and $ c=n\frac{[\omega].[\chi]^{n-1}}{[\chi]^n}. $
Differentiating equation \eqref{j flow} with respect to $t$, we get
$$ \frac{\d}{\d t}\left(\frac{\d \vp}{\d t}\right)=\tilde{\Delta}\left(\frac{\d \vp}{\d t}\right)+\xi\left(\frac{\d \vp}{\d t}\right), $$ where $\tilde{\Delta}f=h^{i\bar{j}}\d_i\d_{\bar{j}}f$ and $h^{i\bar{j}}=\chi^{i\bar{k}}g_{l\bar{k}}\chi^{l\bar{j}}.$
Then the maximum principle implies that $$\min_{X_n}\left.\frac{\d \vp}{\d t}\right|_{t=0} \leq \frac{\d \vp}{\d t}\leq \max_{X_n}\left.\frac{\d \vp}{\d t}\right|_{t=0}. $$
In particular, \[\label{trace upper bound} \La_{\chi_\vp}\omega\leq \max_{X_n}\La_{\chi_0}\omega+\max_{X_n}\theta_\xi(\chi_{\phi})-\min_{X_n}\theta_\xi(\chi_0).  \]
By Corollary 5.3 of \cite{zhu}, $\theta_\xi(\chi_\vp)$ is uniformly bounded. Then the right hand side of the above equation is bounded above by strictly positive constant $A.$
Thus \[\label{eigen.lower bound} \chi_\vp\geq A\omega, \] as long as the flow exists.

Let $ \omega=\ddbar g(s) $ and $ \chi_t=\ddbar f(s,t) $, for $t\geq 0$  satisfy the Calabi Ansatz from the previous section. Then the modified $ J $ flow \eqref{j flow} becomes  \[\label{potflow} \frac{\d f(s,t)}{\d t}=c- (n-1)\frac{g'(s)}{f'(s,t)}-\frac{g''(s)}{f''(s,t)}+kf'(s,t)-\frac{kn}{n+1}\frac{b^{n+1}-1}{b^n-1}. \]

As before, we define strictly increasing functions in the variable $\tau\in [1,b]$ for all $t\geq 0$, that is
$ \psi:[1,b]\times[0,\infty)\to [1,a] $ by $$ \psi(f'(s,t),t)=g'(s). $$

\begin{prop}
	The function $ \psi=\psi(\tau,t) $ defined above is a classical solution of the initial-boundary value problem 
	\begin{equation}\label{flow}\begin{cases}
		\frac{\d \psi}{\d t}=Q(\psi)P[\psi],\\ \psi(1,t)=1,\quad \psi(b,t)=a, \forall t\in (0,\infty) \end{cases}
	\end{equation}
	where 
	\begin{equation}\label{flow}P[\psi] = \psi''(\tau,t)+(n-1)\frac{\psi'(\tau,t)}{\tau}-(n-1)\frac{\psi(\tau,t)}{\tau^2}-k,
	\end{equation} and $ Q\in \textbf{C}^\infty([1,a],[0,\infty)) $ is uniquely determined by $ g $ such that $ Q(x)=0 $ if and only if $ x=1 $ and $ x=b $.
	Moreover, the initial data $ \psi(\tau,0)=\psi_0(\tau) $ is any smooth monotone function satisfying the above boundary conditions. \end{prop}
\begin{proof}
	We have the equation $ \psi(f'(s,t))=g'(s). $ Take derivative with respect to the variable $ t. $
	\[\label{flow1} \frac{\d \psi}{\d \tau}\frac{\d f'}{\d t}+\frac{\d \psi}{\d t}=0  \]
	And if we take derivative of the equation \eqref{potflow} with respect to the variable $ s $, we get
	$$ \frac{\d f'}{\d t}=-(n-1)\left(\frac{g'}{f'}\right)'-\left(\frac{g''}{f''}\right)'+kf''(s,t) $$
	Using this in equation \eqref{flow1}, we get
	\begin{align*}
		\frac{\d \psi}{\d t}&=\psi'(f'(s,t))\left[(n-1)\left(f''\frac{\psi'(\tau)}{f'}-\frac{\psi(\tau)}{f'^2}f''\right)+\psi''(\tau)f''-kf''\right] \\
		&=\psi'(f'(s,t))f''(s,t)\left[(n-1)\left(\frac{\psi'(\tau)}{\tau}-\frac{\psi(\tau)}{\tau^2}\right)+\psi''(\tau)-k\right] \\
		&=g''(s)\left(\psi''(\tau)+(n-1)\frac{\psi'(\tau)}{\tau}-(n-1)\frac{\psi(\tau)}{\tau^2}-k\right)
	\end{align*}
	Since $ g $ is convex and $ g'>0 $, we can get that $ g''(s)=g''((g')^{-1}(\psi)) $.
	Let $ Q(\psi)=g''((g')^{-1}(\psi)). $
	Then by the asymptotic properties of $ g $, we can get that $ Q(1)=Q(b)=0 $ and $ Q(\tau)>0 $, for all $ \tau \in (1,b). $

\end{proof}

\section{Convergence of the flow with the special initial data}
		In this section we prove a weaker version of our main theorem for a special choice of initial condition.
		For the most part, our proof runs parallel to the arguments in \cite{fanglai} and \cite{vrj}. There is
		however one key point of difference that we now wish to explain. Recall that the proof of \cite{fanglai} and \cite{vrj} also proceeds by first proving convergence for a special initial K\"ahler form, namely
		the K\"ahler form for which the function $\psi$ is linear. With this choice of initial data one can then prove a comparison principle along the flow and a monotonicity property along the flow. Unfortunately,  as we explain below (cf. Remark \ref{counter example}), for the modified $J$-flow, the choice of a linear initial function does not quite work. We instead have to choose another concave, strictly increasing function. The next lemma summarises the key properties that we need from the initial function.  
		
\begin{lem}
	\label{initial}For $c_k+k\leq n-1$, there exists a smooth, strictly monotonic and concave function $\psi_0:[1,b]\rightarrow [1,a]$ with the following properties: 
	\begin{enumerate}
		\item  $\psi_0 \geq \sigma $, where we define
		\[\sigma(\tau)=\begin{cases}
			1,&1\leq \tau\leq \la;\\
			\psi_\la(\tau),&\la\leq \tau\leq b,
		\end{cases} \text{ and $\la\in(1,b)$ satisfies } C_k(\la)+k\la=\frac{n-1}{\la}.\]
		\item $P[\psi_0]\leq 0$, where $P$ is the $2^{nd}$ order linear differential operator given by \eqref{flow}.	\end{enumerate}
\end{lem}

\begin{proof}
	We define $\psi_0$ an increasing, smooth and concave function by $$ \psi_0(\tau)=\frac{a-1}{b^{-n}-1}(\tau^{-n}-1)+1. $$		
	
	We know that the $\la$ satisfies the equation $C_k(\la)+k\la=\frac{n-1}{\la}$ this implies that $$ \frac{n+1-k\la^2}{n+1}-\la\frac{C_k(\la)}{n}>0,$$ therefore $ \sigma $ is convex on $ [1,b]. $ Since $\psi_0$ is concave we have $$ 1 \leq \sigma(\tau) < \psi_0(\tau) < a, \text{ for all } \tau \in (1,b). $$
	Next, we compute 
	\begin{align*}
		P[\psi_0]&=\frac{a-1}{b^{-n}-1}n(n+1)\frac{1}{\tau^{n+2}}-(n-1)\frac{a-1}{b^{-n}-1}\frac{n}{\tau^{n+2}}-(n-1) \frac{\psi_0(\tau)}{\tau^2}-k\\
		&=-\frac{2n(a-1)}{1-b^{-n}}\frac{1}{\tau^{n+2}}-(n-1) \frac{\psi_0(\tau)}{\tau^2}-k\\
		&< 0
	\end{align*}
\end{proof}
\begin{prop}\label{funflow}
	Let $0<c_k+k\leq n-1$ and suppose $\psi(\tau,t)$ be the solution of the equation \eqref{flow} with the initial value $\psi_0(\tau)=\psi(\tau,0)$, constructed in the previous lemma, then
		\[ 	\psi(\tau,t)\xrightarrow[C_{loc}^1((\la,b))]{C^0}\sigma(\tau)=
			\begin{cases}
				1, & 1\leq \tau \leq \la \\
				\psi_\la(\tau), & \la \leq \tau \leq b
			\end{cases} , \]		
		where $ \psi_\la $ is a solution of the ODE \eqref{ode} with $ \psi_\la(\la)=1 $ and $ \psi_\la(b)=a $ and $ \la $ can be described as $$ \la=\inf\{s: \text{ there exist } \psi_s \text{ satisfies } \eqref{ode}, \psi_s(s)=1, \psi_s(b)=a \text{ and }\psi_s'(\tau) > 0,\text{ for all }\tau\in(s,b) \}. $$ 
\end{prop}
\begin{proof}		
	The proof consists of several steps.
		\begin{itemize}
		\item First we claim that $\psi(\tau,t)\geq \sigma(\tau).$
		Define $ h(\tau,t)=\psi(\tau,t)-\sigma(\tau) $ and it satisfies the following equation on $[\la,b]$ \[\label{heqn1} \frac{\d h}{\d t}=Q(\psi)\left(h''(\tau,t)+(n-1)\frac{h'(\tau,t)}{\tau}-(n-1)\frac{h(\tau,t)}{\tau^2}\right). \]
		Since $\psi(\tau,t)$ is an increasing function for each $t$, $\psi(\tau,t)>1$, for all $(\tau,t)\in (1,b]\times[0,\infty),$ hence the claim is equivalent to proving that $$ \inf_{[\la,b]\times[0,T]} h(\tau,t) \geq0, \text{ for all } T>0.$$
		Suppose the claim is not true, that is, there exist a $T>0$ such that $\inf_{[\la,b]\times[0,T]}h(\tau,t)<0.$ Let $h(x,t_0)=\inf_{[\la,b]\times[0,T]}h(\tau,t), $ for some $(x,t_0)\in [\la,b]\times[0,T].$ 
		
		Since $h(\tau,0)=\psi_0(\tau)-\sigma(\tau)\geq 0$ (by Lemma \ref{initial}), $h(b,t)=0$ and $h(\la,t)>0$ we can see that $t_0>0$ and $x\in (\la,b).$ And we have $$Q(\psi(x,t_0))>0, \frac{\d h}{\d t}(x,t_0)\leq0, h'(x,t_0)=0 \text{ and } h''(x,t_0)\geq 0 .$$
		Hence, at $(x,t_0)$ we have
		\begin{align*}
			\Bigg(\frac{\d h}{\d t}-Q(\psi)\left(h''+(n-1)\frac{h'}{\tau}-(n-1)\frac{h}{\tau^2}\right)\Bigg)(x,t_0)&\leq (n-1)Q(\psi(x,t_0))\frac{h(x,t_0)}{x^2}\\
			&<0
		\end{align*}
		which is a contradiction to the equation \eqref{heqn1} satisfied by $h$ on $[\la,b].$
		Hence we have $$ \psi(\tau,t)\geq\sigma(\tau), \text{ for all } \tau \in [1,b]. $$
		\item Next we claim that $\psi(\tau,t)$ is decreasing in $t.$ Let $H(\tau,t)=\frac{\d \psi}{\d t}(\tau,t)$.
		Firstly note that by our choice of $\psi_0$ we have that by Lemma \ref{initial} $$H(\tau,0) = Q(\psi_0)P[\psi_0] \leq 0.$$The claim is equivalent to proving that $$\sup_{[1,b]\times[0,T]}H(\tau,t)\leq 0, \text{ for all } T>0.$$	
		Suppose this is not true, then there exist a $T>0$ such that $\sup_{[1,b]\times[0,T]}H(\tau,t)>0.$		
		Let $H(x,t_0)=\sup_{[1,b]\times[0,T]}H(\tau,t)$ and consider $$F(\tau,t)=e^{-Ct}H(\tau,t),\text{ where }C=\sup_{[1,b]\times[0,T]}Q'(\psi(\tau,t))\frac{H(\tau,t)}{Q(\psi(\tau,t))}.$$
		Since $$\frac{H(\tau,t)}{Q(\psi(\tau,t))}=\psi''(\tau,t)+(n-1)\frac{\psi'(\tau,t)}{\tau}-(n-1)\frac{\psi(\tau,t)}{\tau^2}-k,$$ the constant $C$ is finite.
		
		Let $F(y,t_1)=\sup_{[1,b]\times[0,T]}F(\tau,t).$ Since $F(\tau,0)=H(\tau,0)\leq0$ and $ F(1,t)=F(b,t)=0 $ it should be the case that $ t_1>0$ and $y\in (1,b)$. Then at $(y,t_1)$ we have \begin{gather*}
			F(y,t_1)>0,Q(\psi(y,t_1))>0,H(y,t_1)>0, \frac{\d F}{\d t}(y,t_1)\geq 0,\\ F'(y,t_1)=H'(y,t_1)=0,F''(y,t_1)\leq 0 \text{ and } H''(y,t_1)\leq 0.
		\end{gather*}
		This implies that 
		$$	0\leq \Big(\frac{\d F}{\d t}-Q(\psi)F''\Big)(y,t_1)=e^{-Ct_1}\left(-CH(y,t_1)+\frac{\d H}{\d t}(y,t_1)-Q(\psi(y,t_1))H''(y,t_1)\right)$$
		so
		\begin{multline*}
			0\leq -e^{-Ct_1}\Bigg(H(y,t_1)^2\frac{Q'(\psi)}{Q(\psi)}-\frac{\d H}{\d t}(y,t_1)\\+Q(\psi)\Big(H''(y,t_1)+(n-1)\frac{H'(y,t_1)}{y}-(n-1)\frac{H(y,t_1)}{y^2}\Big)+(n-1)Q(\psi)\frac{H(y,t_1)}{y^2}\Bigg)
		\end{multline*}
		Thus
		\begin{multline*}
			H(y,t_1)^2\frac{Q'(\psi)}{Q(\psi)}-\frac{\d H}{\d t}(y,t_1)+Q(\psi)\Big(H''(y,t_1)\\+(n-1)\frac{H'(y,t_1)}{y}-(n-1)\frac{H(y,t_1)}{y^2}\Big)\leq-(n-1)Q(\psi)\frac{H(y,t_1)}{y^2}<0
		\end{multline*}
		which is a contradiction to the fact that $H$ actually satisfies the following equation $$\frac{\d H}{\d t}(\tau,t)=H(\tau,t)^2\frac{Q'(\psi)}{Q(\psi)}+Q(\psi)\Big(H''(\tau,t)+(n-1)\frac{H'(\tau,t)}{\tau}-(n-1)\frac{H(\tau,t)}{\tau^2}\Big). $$
		Hence $$ \frac{\d \psi}{\d t} \leq 0,$$
		and $ \psi(\tau,t) $ is a decreasing function in $t$.
		\item Now we prove the $C^0$ convergence. The proof is very similar to that of \cite{vrj}, so we give only sketch of the proof and for the details we refer the reader to \cite{vrj}. Since $\La_{\chi_{\vp(t)}}\omega=\psi'(\tau,t)+(n-1)\frac{\psi(\tau,t)}{\tau}$ and by the inequality \eqref{trace upper bound} we get that $0<\psi'(\tau,t)< C$ for some positive constant C, together this with the previous claim implies that $\psi(\tau,t)$ converges uniformly on $[1,b]$ to its pointwise limit $\psi_\infty(\tau)$, an increasing continuous function.
		
		\item Next, we prove the $C^1_{loc}$ convergence of $\psi(\tau,t)$ and differentiability of $\psi_\infty$ on $(s,b)$, where $$s=\sup\{x:\psi_\infty(x)=1\}.$$ Since $\psi(\tau,t)\geq \sigma(\tau),$ we get $\psi_\infty(\tau)\geq \sigma(\tau)$.
		Next, as in \cite{vrj}, we consider the evolving point-wise slope $$\eta(\tau,t):=\psi'(\tau,t)+(n-1)\frac{\psi(\tau,t)}{\tau}-k\tau.$$
		Then $$ \frac{\d \psi}{\d t}(\tau,t)=Q(\psi)\eta'(\tau,t). $$ So for each $t$, $\eta(\tau,t)$ is decreasing on $[1,b].$ Let us fix $0<\delta<<1$ and $T>0$ . Since $\psi(\tau,t)$ converges uniformly to $\psi_\infty$ we get $Q(\psi(\tau,t))$ is uniformly bounded below, for all $(\tau,t)\in [s+\delta,b-\delta]\times[T,\infty) $, and this implies  \[ \int_{T}^{\infty}\int_{s+\delta}^{b-\delta}\left|\frac{\d \psi}{\d t}\right|d\tau dt=\int_{s+\delta}^{b-\delta}\Big(\psi(\tau,T)-\psi(\tau,\infty)\Big)d\tau \leq D \]  for some constant $D>0$ independent of $\delta.$ This implies that \[ \int_{T}^{\infty}\Big(\eta(1+\delta,t)-\eta(b-\delta,t)\Big)dt=-\int_{T}^{\infty}\int_{s+\delta}^{b-\delta}\eta'(\tau,t)d\tau dt\leq 
		C(\delta), \] for some $C(\delta)>0$.
		Therefore, there exist a sequence $\{t_j\}$ such that $\eta(\tau,t_j)$ converges to $\eta_\infty$ and by the monotinicity of $\eta$ the limit $\eta_\infty$ will be a constant. And the uniform convergence of $\psi(\tau,t)$ gives the uniform convergence of $\psi'(\tau,t_j)$ on $(s,b).$
		
		Now, let $x,\tau\in (s,b)$, then $$\eta_\infty(\tau-x)=\lim_{j\to \infty}\int_{x}^{\tau}\eta(y,t_j)dy=\psi_\infty(\tau)-\psi_\infty(x)+\int_{x}^{\tau}\Big((n-1)\frac{\psi_\infty(y)}{y}-ky\Big) dy. $$ Hence $\psi_\infty$ is differentiable on $(s,b)$ and it satisfies the following equation $$\psi_\infty'(\tau)+(n-1)\frac{\psi_\infty(\tau)}{\tau}-k\tau=\eta_\infty .$$
		 
		\item  Next we claim that $\lim_{t \to \infty}\psi(\tau,t)=\psi_\infty(\tau)=\sigma(\tau).$
		Since $ \psi_\infty(\tau)\geq \sigma(\tau) $
		and $ 1 \leq \sigma(\tau) < \psi_0(\tau) < a, \text{ for all } \tau \in (1,b) $ the function $ \psi_\infty $ should be of the following form
		$$ \psi_\infty(\tau)=
		\begin{cases}
			1,&1\leq \tau\leq s;\\
			\rho(\tau),&s\leq\tau\leq  b,
		\end{cases} $$
		where $ s\in [1,b) $ and $ \rho $ is a solution of the ODE \eqref{ode} in the interval $ [s,b] $ with the boundary condition $ \rho(s)=1 \text{ and }\rho(b)=a $ and $ \rho'(\tau)>0, \text{ for all }\tau\in (s,b) $ which follows from the last claim and the fact that $\psi_\infty$ is increasing. This implies that $ s\leq \la. $
		From the definition of $ \la $ we can get that $ s=\la. $
		And the uniqueness of the solution of the ODE \eqref{ode} on the interval $[\la,b]$ implies that $$ \psi_\infty(\tau)= \sigma(\tau). $$ Therefore $ \psi(\tau,t) $ converges in $C^1_{loc}((\la,b))$ to its limiting function $ \sigma(\tau)=\lim_{t \to \infty}\psi(\tau,t) $, follows from the fact that $\psi_\infty'(\tau)>0$ on any compact $K\subset(\la,b).$
	\end{itemize}
	Finally, since $\psi_\infty$ satisfies the ODE \eqref{ode} on $(\la,b)$, with the boundary conditions $\psi_\infty(\la)=1$ and $\psi_\infty(b)=a$, we get $\eta_\infty=C_k(\la).$ 
\end{proof}
\begin{rem}\label{counter example}
	We now explain why the straight line path does not work for the general
	modified $J$ flow. The main issue that the sign of $P[\psi_0]$ could be indeterminate, which is a key starting point in establishing the monotonicity of $\psi(\tau,t)$ in the time variable. For instance,
	consider the following example on $X_2$, let $a=12, b=6$ and $k=1$, then in this case $0<c_k+k<1$ and for the straight line path $\psi_0$ we have $P[\psi_0](1)>0$ and $P[\psi_0](6)<0.$ 
\end{rem}
\begin{rem}
	For the case $c_k+k>n-1,$ the solution of the ODE \eqref{ode}, $\tilde{\psi}$ can be neither convex nor concave. For example, on $X_3$, if we take $a=4,b=2$ and $k=1$, then the solution $\tilde{\psi}$ is neither convex nor concave. And we can not choose the straight line path as  a intial value of the flow, as $\tilde{\psi}$ neither lies above nor lies below the straight line path completely. So, we choose $\psi_0(\tau)=\tilde{\psi}+\frac{(\tau-1)(b-\tau)}{N\tau^n},$ for some sufficiently large $N,$ as the initial data of the flow, then $\psi_0$ satisfies all the properties mentioned in Lemma \ref{initial} and then we can proceed as in the last propostion to prove the convergence in the case of $c_k+k>n-1.$
\end{rem}
\section{Uniqueness and higher order estimates}
Uniqueness and higher order estimates for the $J$-flow with the Calabi ansatz in the unstable case is done in \cite{vrj}. In this section we derive the uniform higher order estimates and prove the uniqueness of the limit of the modified $J$-flow for the cases $c_k+k\leq n-1,$ for this we need some results which are very similar to the results proved in \cite{vrj} with necessary modifications. So, just for the readers convenience we state those results and give sketch of the proof for some of the results which needed the necessary modifications.

In the following proposition we prove that on a given compact subset $K \subset \RR$ the normalized ansatz K\"ahler potentials actually converge. Let $$ \Phi:[1,a]\times[0,\infty)\to [1,b],\text{ defined by } \Phi(\tau,t):=\psi^{-1}(\tau,t). $$ That is $f'(s,t)=\Phi(g'(s),t)$ and we define $\Phi_\infty:=(\psi_\infty|_{[\la,b]})^{-1}:[1,a]\to[\la,b]. $
\begin{prop}
	\label{potential convergence}For any compact $K\subset\RR$, $$f'(s,t)\to (\Phi_\infty\circ g')(s), $$ increasingly in $L^\infty(K)$ as $t\to \infty.$
\end{prop}
\begin{proof}
	For a given compact subset $K\subset\RR$ we can find a $B\subset \subset (\la,b)$ such that $K=(\Phi_\infty\circ g')^{-1}(B)$, and we have $\psi_\infty'(\tau)>0, $ for all $\tau\in B.$ Then by Proposition \ref{funflow}, there exists $C,T>0$ such that for any $t>T$ and $\tau\in B$, we have $$C^{-1}<\psi'(\tau,t)<C. $$
	Since $\psi$ decreases in $t$ and by definition of $\Phi$ we get that $\Phi(.,t)$ converges uniformly to $\Phi_\infty$ on $\Phi_\infty^{-1}(B)$ increasingly in $t$. Thus, for any compact $K\subset\RR$, $f'(s,t)$ converges increasingly to $(\Phi_\infty\circ g')(s)$, for $s\in K.$
\end{proof}
We define the normalized ansatz potentials by $$\hat{f}(s,t):=f(s,t)-f(0,t)=\int_{0}^{s}f'(r,t)dr$$ and $$f_\infty(s):=\int_{0}^{s}(\Phi_\infty\circ g')(r)dr.$$
By the above proposition we get that $\hat{f}(s,t)$ converges uniformly to $f_\infty(s)$ on any compact set $K\subset\RR$ as $t\to\infty.$

\begin{lem}
	\label{hatf}For all $t\geq 0,$ we have $$\hat{f}(s,t)\leq f_\infty(s) \text{ when } s\geq 0, \text{ and } \hat{f}(s,t)\geq f_\infty(s) \text{ when } s\leq 0. $$
\end{lem}
\begin{proof}
	Since $\hat{f}(s,t)-f_\infty(s)=\int_{0}^{s}(f'(r,t)-(\Phi_\infty\circ g')(r))dr$, the lemma follows from the Proposition \ref{potential convergence}.
\end{proof}
Let $$ \phi(s,t):=f(s,t)-f_0(s), \phi_\infty(s):=f_\infty(s)-f_0(s), $$ where $f_0$ is the global potential for $\chi_0=\ddbar f_0(s) \in b[E_\infty]-[E_0]$ so $f'_0(-\infty)=1$ and $f_0'(\infty)=b$ also since, $f_\infty$ satisfies the modified $J$ equation on $[\la,b]$ with $f'_\infty(-\infty)=\la$ and $f'_\infty(\infty)=b$, $\ddbar f_\infty\in b[E_\infty]-\la[E_0].$

Therefore $\phi(\cdot,t)$ is a K\"ahler potential with respect to the metric $\chi_0$, for all $t\geq 0$ and it satisfies the following equation $$\frac{\d \phi}{\d t}(s,t)=c_k- (n-1)\frac{g'(s)}{f'(s,t)}-\frac{g''(s)}{f''(s,t)}+kf'(s,t) $$
Let $\nu:\RR\to (-\infty,0]$ be a smooth increasing function satisfying 
\begin{gather*}
	\nu(s)=(\la-1)s \text{ when } s\in (-\infty,-1) \text{ and }\\
	\nu(s)=0 \text{ when } s\in (0,\infty).
\end{gather*}
\begin{lem}
	\label{hatp}Let $$ \hat{\phi}_\infty:=\phi_\infty-\nu, $$  then $\hat{\phi}_\infty$ is bounded on $\RR.$
\end{lem}
\begin{proof}
	By definition $\hat{\phi}_\infty=f_\infty-(f_0+\nu)$ and $\ddbar f_\infty\in b[E_\infty]-\la[E_0]$ is a K\"ahler current and $\ddbar(f_0+\nu)$ is a K\"ahler metric, thus $\hat{\phi}_\infty$ is bounded in $\RR.$
\end{proof}
\begin{lem}
	\label{phinu}There exist a $C>0$ such that for all $(s,t)\in \RR\times[0,\infty)$, we have $$\nu(s)-C\leq \phi(s,t)-\phi(0,t)\leq C $$ and $$\lim_{t\to \infty}\frac{\d \phi}{\d t}(0,t)=c_k-C_k(\la). $$
\end{lem}
\begin{proof}
		Without loss of generality we can assume $f_0(0)=0$ and by the definition of $\phi(s,t)$, we have 
	$$\phi(s,t)-\phi(0,t)=f(s,t)-f_0(s)-[f(0,t)-f_0(0)]=\hat{f}(s,t)-f_0(s). $$ and the first equation of the lemma follows from Lemma \ref{hatf} and Lemma \ref{hatp}. And by Proposition \ref{potential convergence} and Proposition \ref{funflow} for any $K\subset \subset \RR$, we have 
	$$\lim_{t\to \infty}\left|\frac{g''}{f''}+(n-1)\frac{g'}{f'}-kf'-C_k(\la)\right|_{L^\infty(K)}=0. $$
	Hence $$\lim_{t\to \infty}\frac{\d \phi}{\d t}(0,t)=c_k-C_k(\la). $$
\end{proof}
We chose $\phi(s,t)$ as a solution of the flow with the special initial metric, now let $\vp$ be the solution of the modified $J$ flow with any initial data $\vp(.,0)$ with Calabi symmetry. And let $$\vartheta(s,t)=\vp(s,t)-f_0(s) \text{ and } \chi_0=\ddbar f_0.$$
Then $\vartheta(.,t)$ are K\"ahler potentials of $[\chi_0]$ and the modified $J$ flow is $$\frac{\d \vartheta}{\d t}=c+\theta_\xi(\chi_\vartheta)-\La_{\chi_{\vartheta}}\omega. $$
\begin{lem}
	\label{thest}There exists a $C>0$ such that on $X_n\times[0,\infty)$, $$\phi_\infty(s)-C\leq \vartheta(s,t)-\phi(0,t)\leq C. $$
\end{lem}
\begin{proof}
	We know that $ \theta_\xi(\chi_\vp)=\theta_\xi(\chi)+\xi(\vp). $
	If we take difference between the flows with initial values $ \phi(\cdot,0) $ and $ \vartheta(\cdot,0) $, then $ \phi(x,t)-\vartheta(x,t) $ satisfies the following equation 
	$$ \frac{\d}{\d t} \left(\phi(x,t)-\vartheta(x,t)\right)=F^{i\bar{j}}[(1-s_t)\chi_{\phi(t)}+s_t\chi_{\vartheta(t)}](\phi(t)-\vartheta(t))_{i\bar{j}}+\xi(\phi(t)-\vartheta(t)), $$
	where $ 0<s_t<1. $ By the maximum principle we have $$ |\phi(x,t)-\vartheta(x,t)|_{C^0}\leq |\phi(x,0)-\vartheta(x,0)|_{C^0}\leq C, \text{ for all }t \in [0,\infty). $$ This together with the last lemma we can see that $$\nu(s)-C_1\leq \vartheta(s,t)-\phi(0,t)\leq C_1. $$ And by Lemma \ref{hatp} we can get that $\nu(s)-C_2\leq \phi(s,t)\leq \nu(s)+C_2$ this implies that $$\phi_\infty(s)-C\leq \vartheta(s,t)-\phi(0,t)\leq C. $$ 
\end{proof}
The K\"ahler potential $\phi_\infty$ of $\chi_0$ is not a strict subsolution of the modified $J$ equation for the new K\"ahler classes $[\chi_\infty]=b[E_\infty]-\la[E_0] $ and $[\omega]$ and the constant $C_k(\la),$ because $\chi_\infty$ blows up in the normal direction to $E_0.$

Let $f_\infty$ be the limiting solution corresponding to the new K\"ahler classes. Then $f_\infty'(-\infty)=\la>1$ and $$\frac{g''(s)}{f_\infty''(s)}+(n-1)\frac{g'(s)}{f_\infty'(s)}-kf_\infty'(s)=C_k(\la) $$
this implies that $$(n-2)\frac{g'(s)}{f_\infty'(s)}+\frac{g''(s)}{f_\infty''(s)}-kf_\infty'(s)\leq C_k(\la)-\frac{1}{b}, ~(n-1)\frac{g'(s)}{f_\infty'(s)}-kf_\infty'(s)\leq C_k(\la) $$
and the equality holds at $-\infty$ in the second inequality as $$\lim_{s\to -\infty}e^{-s}f_\infty''(s)=\infty. $$
Our aim is to get a strict subsolution so we modify $f_\infty$ using the barrier function $\nu$ which is increasing and smooth and defined in the following way $$\nu(s)=s \text{ when } s\leq -1 \text{ and } \nu(s)=0 \text{ when } s>0. $$

Let $$V_\infty(s)=f_\infty(s)+\vep\nu(s) $$ for sufficiently small $\vep>0.$ Then $V_\infty$ defines a K\"ahler metric $\Omega=\ddbar V_\infty\in b[E_\infty]-(\la+\vep)[E_0].$
Further, we can choose a constant $d>0$ so that \[ \begin{split}
	\label{strict subsoln}(n-2)\frac{g'(s)}{V'_\infty(s)}+\frac{g''(s)}{V''_\infty(s)}-kV_\infty'(s)<&C_k(\la)-2d\vep \\ \text{ and }\quad (n-1)\frac{g'(s)}{V'_\infty(s)}-kV_\infty'(s)<& C_k(\la)-2d\vep.
\end{split} \]
We will use this subsolution to prove the second order estimate, as Lemma \ref{2nd order est}.

Now, let $$ \hat{\vartheta}(s,t):=\vartheta(s,t)-\phi(0,t). $$
Then \begin{align*}
	 \frac{\d \hat{\vartheta}}{\d t}=&c+\theta_\xi(\chi_{\hat{\vartheta}(t)})-\La_{\chi_0+\ddbar\hat{\vartheta}}(\omega)-\frac{\d \phi}{\d t}(0,t)\\ =&C_k(\la)+\delta(t)+k\vartheta'(s,t)+kf_0'(s)-\La_{\chi_0+\ddbar\hat{\vartheta}}(\omega),
\end{align*} where $\delta(t):=c_k-C_k(\la)-\frac{\d \phi}{\d t}(0,t)\to 0$
as $t\to \infty$ by Lemma \ref{phinu}.

Let $$u:=\hat{\vartheta}-(V_\infty-f_0). $$ Then $ u(s,t)=\hat{\vartheta}(s,t)-(f_\infty(s)+\vep\nu(s)-f_0(s))=\hat{\vartheta}(s,t)-(\phi_\infty(s)-\vep\nu(s)), $
thus $u$ is smooth on $X_n\smallsetminus E_0$ and tends to $\infty$ along $E_0.$

\begin{lem}
	\label{low. bound barrier}There exist a constant $C>0$ such that $$ \inf_{X_n\times[0,\infty)}u>-C. $$ 
\end{lem}
\begin{proof}
	By Lemma \ref{thest} we have
	
	$ u(s,t)=\hat{\vartheta}(s,t)-(\phi_\infty(s)-\vep\nu(s))\geq \phi_\infty(s)-C-(\phi_\infty(s)-\vep\nu(s))\geq -C .$
\end{proof}
The following lemma is analogous to Lemma 3.1 of Song-Weinkove in \cite{swflow} which guarantee us the second order estimates away from the exceptional divisor.
\begin{lem}
	\label{2nd order est}Let $\chi=\chi_0+\ddbar \vartheta$ be the K\"ahler metrics corresponding to the solution of the $J$ flow. Then there exist $A,C>0$ such that on $X_n\smallsetminus E_0,$ we have $$\La_\omega\chi\leq Ce^{Au} .$$ 
\end{lem}
\begin{proof}
	As in \cite{vrj}, we can apply the maximum principle argument to $H=\log\La_\omega\chi-Au$ with the perturbed heat operator and then use the subsolution construction to get the required estimate.

	The evolution for $\log\La_\omega\chi$ is given by the following general estimate $$\left(\frac{\d }{\d t}-\Delta_t\right)\log\La_\omega\chi\leq -\frac{n}{\La_\omega\chi}\left(h^{k\bar{l}}R_{k\bar{l}}^{i\bar{j}}\chi_{i\bar{j}}-\chi^{k\bar{l}}R_{k\bar{l}}\right), $$ where $h^{i\bar{j}}=\chi^{i\bar{l}}\omega_{k\bar{l}}\chi^{k\bar{j}},\Delta_t:=n^{-1}h^{i\bar{j}}\d_j\dbar_j,$ and $R$ is the curvature tensor for $\omega.$

Also 
\begin{align*}
	\left(\frac{\d}{\d t}-\xi-\Delta_t\right)u=&\frac{\d \hat{\vartheta}}{\d t}-\xi(\hat{\vartheta})+\xi(V_\infty-f_0)-\Delta_tu\\
	=&\eta_\infty+\delta(t)-\La_{\chi_0+\ddbar\hat{\vartheta}}(\omega)+k\vartheta'+kf_0'-k\vartheta'\\
	&+kV_\infty'-kf_0'-h^{i\bar{j}}\vartheta_{i\bar{j}}+h^{i\bar{j}}\Omega_{i\bar{j}}-h^{i\bar{j}}f_{0_{i\bar{j}}}\\
	=&\eta_\infty+\delta(t)-2\chi^{i\bar{j}}\omega_{i\bar{j}}+\chi^{i\bar{l}}\omega_{k\bar{l}}\chi^{k\bar{j}}\Omega_{i\bar{j}}+kV_\infty'
\end{align*}
Suppose the maximum of $H$ on $X_n\times[0,t_0]$ is achieved at $(p_0,t_0),$ with $t_0>0$, otherwise the estimate follows trivially. Then applying maximum principle, we have at $(p_0,t_0)$
\begin{multline}
	\label{max principle}0\leq \left(\frac{\d}{\d t}-\xi-\Delta_t\right)H\\
	\leq-(n\La_\omega\chi)^{-1}(h^{k\bar{l}}R_{k\bar{l}}^{i\bar{j}}\chi_{i\bar{j}}-\chi^{k\bar{l}}R_{k\bar{l}}+n^{-1}\xi(\La_\omega\chi))\\
	-A(\eta_\infty+\delta(t)-2\chi^{i\bar{j}}\omega_{i\bar{j}}+\chi^{i\bar{l}}\omega_{k\bar{l}}\chi^{k\bar{j}}\Omega_{i\bar{j}}+kV_\infty')
\end{multline}
Since $\La_\omega\chi$ is uniformly bounded below(by inequality \eqref{eigen.lower bound}), for a given $2d\vep>0$ we can choose sufficiently large $A$ such that at $(p_0,t_0),$ we have 
$$ \eta_\infty+\delta(t)+h^{i\bar{j}}\Omega_{i\bar{j}}-2\chi^{i\bar{j}}\omega_{i\bar{j}}+kV_\infty'\leq -(nA\La_\omega\chi)^{-1}(h^{k\bar{l}}R_{k\bar{l}}^{i\bar{j}}\chi_{i\bar{j}}-\chi^{k\bar{l}}R_{k\bar{l}}-n\xi(\La_\omega\chi))<2d\vep $$
As $\delta(t)\to 0$ when $t\to \infty$, we can find a $T$ such that $|\delta(t)|<d\vep,$ for all $t>T.$ And, we can take $t_0>T$, otherwise, the estimate follows trivially. This implies that \[ \label{large a} \eta_\infty+h^{i\bar{j}}\Omega_{i\bar{j}}-2\chi^{i\bar{j}}\omega_{i\bar{j}}+kV_\infty'< d\vep \]
From \eqref{strict subsoln} we already have that $$ 2d\vep<\eta_\infty-(n-2)\frac{g'}{V_\infty'}-\frac{g''}{V_\infty''}+kV_\infty'\text{ and }2d\vep<\eta_\infty-(n-1)\frac{g'}{V_\infty'}+kV_\infty'. $$
Note that $[h^{i\bar{j}}]=diag\left[\frac{g'}{(f')^2},\dots,\frac{g'}{(f')^2},\frac{g''}{(f'')^2}\right]$, if we apply this in \eqref{max principle} and using \eqref{large a} and the above inequality we get
\begin{align*}
	d\vep\geq& \eta_\infty+(n-1)\frac{g'V_\infty'}{(f')^2}+\frac{g''V_\infty''}{(f'')^2}-2(n-1)\frac{g'}{f'}-2\frac{g''}{f''}+kV_\infty'\\
	=&(n-1)g'V_\infty'\left(\frac{1}{f'}-\frac{1}{V_\infty'}\right)^2+\frac{g''V_\infty''}{(f'')^2}-2\frac{g''}{f''}+\left(\eta_\infty-(n-1)\frac{g'}{V_\infty'}+kV_\infty'\right)\\
	\geq&2d\vep-2\frac{g''}{f''}
\end{align*}
and 
\begin{align*}
	d\vep\geq&\eta_\infty+(n-1)\frac{g'V_\infty'}{(f')^2}+\frac{g''V_\infty''}{(f'')^2}-2(n-1)\frac{g'}{f'}-2\frac{g''}{f''}+kV_\infty'\\
	=&(n-2)g'V_\infty'\left(\frac{1}{f'}-\frac{1}{V_\infty'}\right)^2+g''V_\infty''\left(\frac{1}{f''}-\frac{1}{V_\infty''}\right)^2+\frac{g''V_\infty''}{(f'')^2}-2\frac{g'}{f'}\\
	&\qquad+\left(\eta_\infty-(n-2)\frac{g'}{V_\infty'}-\frac{g''}{V_\infty''}+kV_\infty'\right)\\
	\geq&2d\vep-2\frac{g'}{f'}
\end{align*}
Combinig the above estimates, we have at $(p_0,t_0)$, $$\La_\omega\chi =(n-1)\frac{f'}{g'}+\frac{f''}{g''}\leq \frac{2n}{d\vep} $$ and so for any $p\in X_n$ and $t\in [0,t_0]$ we have 
\begin{align*}
	\log\La_\omega\chi(p,t)=&H(p,t)+Au(p,t)\\
	\leq&\log\La_\omega\chi(p_0,t_0)-Au(p_0,t_0)+Au(p,t)\\
	\leq&-\log d\vep+\log 2n+Au(p,t)-A\inf_{X_n}u(p,t_0).
\end{align*}

	Taking exponential on both sides and using the Lemma \ref{low. bound barrier} we get the required inequality.
\end{proof}
\begin{lem}
	\label{high.est}For any $K\subset \subset X_n\smallsetminus E_0$ and $l>0,$ there exists $C:=C(K,l)$ such that $$\|\hat{\vartheta}\|_{C^l(K)}\leq C. $$
\end{lem}
\begin{proof}
	In section 3 of \cite{swflow}, Song and Weinkove derived the $ C^0 $ estimate by sequence of lemmas using just the $ C^2 $ estimate and not any equation. As we have $ C^2 $ estimate in Lemma \ref{2nd order est}, using the arguments from \cite{swflow} we can derive the $ C^0 $ estimate.
	By Schauder estimates and Evans-Krylov estimate we can obtain the uniform $ C^{2,\al} $ estimate and by bootstrapping we get the higher order estimates. 
\end{proof}
We got all the  necessary estimates to pass the limit, now we prove in the next theorem that irrespective of the initial value the flow always converges to a unique limiting solution.
\begin{thm}
	\label{unique}Let $\vartheta(t)$ be the solution of the modified $J$-flow, then $ \La_{\chi_{\vartheta(t)}}\omega-\theta_\xi(\chi_{\vartheta(t)}) $ converges to $C_\la$ smoothly on $X_n\smallsetminus E_0$ as $t\to \infty,$ where $C_\la=n\frac{ab^{n-1}-\la^{n-1}}{b^n-\la^n}.$
\end{thm}
\begin{proof}
		Let us define a modified $J$-energy functional as $$ E(t)=\int_{X_n}\left(\La_{\chi_{\vartheta(t)}}\omega-\theta_\xi(\chi_{\vartheta(t)})\right)^2\chi_{\vartheta(t)}^n. $$
	In \cite{lishi2} Li-Shi proved that the energy decreases along the modified $J$-flow by showing that the energy functional along the flow satisfies the following equation: $$ \frac{d}{dt}E(t)=-2\int_{X_n}\left|\nabla^{\chi_{\vartheta(t)}}\left(\La_{\chi_{\vartheta(t)}}\omega-\theta_\xi(\chi_{\vartheta(t)})\right)\right|_\omega^2\chi_{\vartheta(t)}^n $$
	The proof follows from the similar arguments of the proof of the result in \cite{vrj}.
	Hence, the modified $J$-flow always converges to a unique limit and it satisfies the modified $J$-equation in the new K\"ahler class.
\end{proof}
\section{Proof of the main theorem}
\begin{proof}[\textbf{Proof of Theorem \ref{mt1}}]
\begin{case}
	It follows from the Case 1. of Theorem \ref{funflow}. that the solution of the ODE \eqref{ode} $ \tilde{\psi} $ satisfies 
	$ \tilde{\psi}'(\tau)>0, \text{ for all } \tau\in [1,b]. $	
	By Corollary \ref{iffcondition} that the condition $c_k+k>n-1$ is equivalent to $$(c+\theta_\xi(\chi))\chi^{n-1}-(n-1)\omega\wedge\chi^{n-2}>0.$$ Then by Theorem 3.3 of Li-Shi \cite{lishi2}, the modified $J$-flow converges smoothly to the unique solution of the modified $J$-equation irrespective of the initial value we choose.
	Thus $ \chi_t\to \chi_\infty $ as $ t\to \infty $ smoothly on $ X_n. $ 
	Therefore we get $$ n\frac{\omega\wedge \chi_\infty^{n-1}}{\chi_\infty^n}=c+\theta_\xi(\chi_\infty). $$
\end{case}
\begin{case}	
	And by the arguments similar to the proof of Lemma \ref{high.est} will give us the uniform higher order estimates on compact set $K\subset X_n\smallsetminus E_0$.	So the flow converges to the critical equation away from $ E_0. $ This implies that on $ X_n\smallsetminus E_0 $ we have $$ n\frac{\omega\wedge \chi_\infty^{n-1}}{\chi_\infty^n}=c+\theta_\xi(\chi_\infty). $$ 
	
	Since the convergence of $f(\cdot,t)$ to $f_\infty$ as $t\to \infty$ and it derivatives are uniform on compact subsets $K\subset X_n\smallsetminus E_0$, the function $f_\infty$ is smooth away from the exceptional divisor $E_0$ and it is continuous on $X_n$ and $ f'_\infty(s)=\psi^{-1}(g'(s)) $. This implies that $$ \chi_\infty=\ddbar f_\infty=\chi_0+\ddbar \phi_\infty, \text{ where } \phi_\infty=f_\infty-f_0\in L^\infty(X_n).$$ And $\phi_\infty $ is smooth away from the exceptional divisor.
	In this case we can extend the metric $ \chi_\infty $ to $ E_0. $ And Theorem \ref{unique} gives us that the limiting solution is unique.
\end{case}

\begin{case}
	In this case we have $$ \lim_{s \to -\infty}f'_\infty(s)=\la>1. $$
	This implies that $ \hat{\chi}_\infty=\ddbar f_\infty\in b[E_\infty]-\la[E_0]. $
	
	So $ \chi_t \to \hat{\chi}_\infty+(\la-1)[E_0] $  as $ t\to \infty. $
	And $ f'_\infty(s)=\psi^{-1}(g'(s)) $ where $ \psi $ is a solution of the ODE \eqref{ode} in the interval $ [\la,b] $ with $ \psi'(\la)=0. $
	By the argument similar to Case 2. we get the higher order estimates away form $ E_0. $ So the flow converges to the critical equation away from $ E_0. $ This implies that on $ X_n\smallsetminus E_0 $ we have $$ n\frac{\omega\wedge \hat{\chi}_\infty^{n-1}}{\hat{\chi}_\infty^n}=n\frac{ab^{n-1}-\la^{n-1}}{b^n-\la^n}+\theta_\xi(\hat{\chi}_\infty). $$ And by Proposition \ref{funflow} the $\la\in (1,b)$ is unique and satisfies the equation $$n\frac{ab^{n-1}-\la^{n-1}}{b^n-\la^n}-\frac{nk}{n+1}\frac{b^{n+1}-\la^{n+1}}{b^n-\la^n}+k\la=\frac{n-1}{\la}. $$ And, the uniqueness of limiting solution follows from Theorem \ref{unique}.
\end{case}
\end{proof}
\begin{rem}
	Suppose we are given with two K\"ahler classes $\chi\in b[E_\infty]-b_0[E_0]$, $\omega\in a[E_\infty]-a_0[E_0]$ and the vector field $\xi_k=kw\frac{\d}{\d w} (k\geq 0)$, then we consider the new K\"ahler classes $\frac{b}{b_0}[E_\infty]-[E_0]$ and $\frac{a}{a_0}[E_\infty]-[E_0]$ and let $f_1(s,t)$ be the solution of the modified $J$-flow for the scaled classes with $\ddbar g_1(s)\in \frac{a}{a_0}[E_\infty]-[E_0].$
	
	As the modified $J$-equation has the normalized Hamiltonian on its right side we have to scale the vector field by $\frac{b_0^2}{a_0}$ that is $\xi'=\frac{b_0^2}{a_0}\xi.$
	Then the usual parabolic scaling gives us the flow for the original K\"ahler classes. That is if we take $$ f(s,t)=b_0f_1(s,b_0^{-2}a_0t) \text{ and } g(s)=a_0g_1(s), $$ then this satisfies the modified $J$-flow on the K\"ahler classes $\chi$ and $\omega.$
\end{rem}

\subsection*{Acknowledgments:} The author would like to thank his PhD. advisor Ved Datar for suggesting this problem, for the helpful discussions, and the suggestions for improving the first draft of the current work. He would also like to thank Professor Vamsi P. Pingali for useful comments on the first draft. This research was supported by an NBHM fellowship and the graduate program of the Indian Institute of Science.

\end{document}